\documentclass[10pt,letterpaper]{article}
\makeatletter
\usepackage{amscd}
\usepackage{amsmath}
\usepackage{latexsym}
\usepackage{amsfonts}
\usepackage{amssymb}
\usepackage{amsthm}
\usepackage{graphicx}
\usepackage[dvipdfm,
            pdfstartview=FitH,
            CJKbookmarks=true,
            bookmarksnumbered=true,
            bookmarksopen=true,
            colorlinks=true, 
            pdfborder=001,   
            citecolor=magenta,
            linkcolor=blue,
            ]{hyperref}       

\newtheorem{thm}{\textbf Theorem}[section]
\newtheorem{lem}{\textbf Lemma}[section]
\newtheorem{rem}{\textbf Remark}[section]

\newcommand{\be}{\begin{eqnarray}}
\newcommand{\ee}{\end{eqnarray}}
\newcommand{\mr}{\mathbb{R}}

\newcommand{\mx}{\mbox}

\newcommand{\bes}{\begin{eqnarray*}}
\newcommand{\ees}{\end{eqnarray*}}

\begin{document}
\begin{titlepage}
\title{Global existence and decay rate of strong solution to incompressible Oldroyd type model equations}
\author{Baoquan Yuan\thanks{Corresponding Author: B. Yuan}\ \  and Yun Liu
       \\ School of Mathematics and Information Science,
       \\ Henan Polytechnic University,  Henan,  454000,  China.\\
        (bqyuan@hpu.edu.cn, liuyunlexie@163.com)
          }
\date{}
\end{titlepage}
\maketitle
\begin{abstract}
This paper investigates the global existence and the decay rate in time of a solution to
the Cauchy problem for an incompressible Oldroyd model with a deformation tensor damping term.
There are three major results. The first is the global existence of the solution for small initial data.
Second, we derive the sharp time decay of the solution in $L^{2}-$norm. Finally, the sharp time decay of the solution of higher order Sobolev norms is obtained.
\end{abstract}

\vspace{.2in} {\bf Key words:}\quad Incompressible Oldroyd model; damping term; global existence; decay rate.



\section{Introduction}
\setcounter{equation}{0}

\qquad In this paper, we consider the incompressible Oldroyd model with a deformation tensor damping term
 \begin{eqnarray}\begin{cases}\label{equation}
 \partial_{t}u  -\mu\Delta u + u \cdot \nabla u  + \nabla p = \nabla\cdot(FF^{T}),\\
 \partial_{t}F + \nu F+u \cdot \nabla F  =\nabla uF,\\
 \mbox{div}u=0
 \end{cases}
 \end{eqnarray}
for any $t>0$, $x\in \mathbb{R}^{3}$, where $u=u (t,x)$ is the velocity
of the flow, $\mu>0$ the kinematic viscosity, $\nu\geq 0$ a constant, $p$ the scalar pressure and $F$ the deformation tensor of the fluid.
 We define $(\nabla\cdot F)_{i}=\partial_{x_j}F_{ij}$ for the matrix $F$. When $\nu=0$, the equation (\ref{equation}) reduces to the classic Oldroyd model which exhibits an incompressible non-Newtonian fluid. Many hydrodynamic behaviors of the
complex fluids can be regarded as a consequence of the interaction between the fluid motions and the internal elastic properties.
Physical background on this model can be found in \cite{L}, \cite{1} and \cite{2}.

Supposing $\mbox{div}F^{T}(0,x)=0$, it can be proved that $\mbox{div} F^T(t,x)=0$ a.e. for any time $t>0$. In fact from $(\ref{equation})_{2}$ one has
 \be\label{1.2}
\partial_{t}(\nabla\cdot F^{T})+\nu\nabla\cdot F^{T} +u\cdot\nabla(\nabla\cdot F^{T})=0.
 \ee
 Multiplying the equation (\ref{1.2}) by $\nabla\cdot F^{T}$ and integrating over $\mathbb{R}^{3}$ then using the divergence free condition of $u$, it yields that
  \be
\frac{\mbox{d}}{\mbox{d}t}\|\nabla\cdot F^{T}\|_{L^{2}}^{2}+2\nu\|\nabla\cdot F^{T}\|_{L^{2}}^{2}=0,\nonumber
 \ee
which implies $\|\nabla\cdot F^T\|_{L^2}=0$ for any time $t>0$.
Therefore $\nabla\cdot(FF^{T})=(F_{.i}\cdot\nabla)F_{.i}$, and the system (\ref{equation}) can be written in a equivalent form
 \be\label{1.1}
 \begin{cases}
 \partial_{t}u  - \mu\Delta u + u \cdot \nabla u  + \nabla p = F_{.i}\cdot \nabla F_{.i} ,\\
 \partial_{t}F_{.j} +\nu F_{.j}+ u \cdot \nabla F_{.j}  =F_{.j}\cdot\nabla u, \  j = 1, \cdots , n,\\
\mbox{div}u=0, \mbox{div}F^{T}=0.
 \end{cases}
 \ee

If $\nu=0$, (\ref{1.1}) is the classical incompressible Oldroyd model equations. For this model equations, the first thing concerned is the existence of the local or global solution. Lin, Liu and Zhang \cite{L-L-ZHANG} proved the local existence of smooth solutions and the global existence of classical solutions with small initial data in both the whole space and the periodic domain, if the initial data is sufficiently close to the equilibrium state for the global existence case. Later, Lei, Liu and Zhou \cite{L-L-ZHOU} established the similar existence result of both local and global smooth solutions to the Cauchy problem of incompressible Oldroyd model equations provided that the initial data is sufficiently close to the equilibrium state.

{\bf Theorem  A} For the divergence free smooth initial data
$({u_0},{F_0}) \in {H^2}({\mathbb{R}^n})$ for $n=2$ or $3$, there exists a positive
time $T=T(\|u_{0}\|_{H^{2}},\|F_{0}\|_{H^{2}})$ such that the system
(\ref{1.1}) with $\nu=0$ possesses a unique smooth solution on
 $[0, T]$ with
  \bes
  u \in {L^\infty }([0,T];{H^2}({\mathbb{R}^n})) \cap
 {L^2}([0,T];{H^3}({\mathbb{R}^n})),\  F \in{L^\infty }([0,T];{H^2}({\mathbb{R}^n})).
 \ees
Moreover, if $T^{*}$ is the maximal time of existence, then
\bes
 \int_0^{{T^*}} {\left\| {\nabla u} \right\|}
_{{H^2}}^2\mathrm{d}s =  + \infty.
 \ees

 In a bounded domain, Lin and Zhang \cite{L-ZHANG} showed the local well-posedness of the initial-boundary value problem of the Oldroyd model with Dirichlet condition and the global well-posedness of the initial-boundary value problem when the initial data is sufficiently close to the equilibrium state. Qian \cite{jq} obtained the local existence of the solution with initial data in critical Besov space, and if the initial data is sufficiently close to the equilibrium state in the critical Besov, the solution is globally in time. For more studies on the topics of the Oldroyd model readers refer to \cite{A,B,C,D}.

Recently, we \cite{Yuan-Liu} establish a local well-posedness result in $H^s(\mr^3)$ for $s>\frac32$ for the classical incompressible Oldroyd model equations by virtue of a new commutator estimate proved by Fefferman etc. \cite{Feff}. That is

{\bf Theorem  B} Assume $u_{0},F_{0}\in H^{s} (\mathbb{R}^{3})$ with $s>\frac{3}{2}$.
 Then, there exists a time $T=T(\|u_{0}\|_{H^{s}},\|F_{0}\|_{H^{s}})>0$ such that equations (\ref{1.1}) with $\nu=0$ have a unique strong solution $(u,F)$
 with $u,F\in C([0,T]; H^{s} (\mathbb{R}^{3}))$.

This paper is dedicated to the study of the Cauchy problem for system ({\ref{1.1}) with the initial condition
\be\label{1.4}
(u,F)(0,x)=(u_{0}(x),F_{0}(x))\in H^m(\mr^3) \mbox{ for } m\ge 3.
\ee

The purpose of this paper is to obtain the global existence of small initial datum, and the decay rate of the smooth solution for the model (\ref{1.1}).
For the system (\ref{1.1}) with $\nu=0$, the local-in-time existence and uniqueness of solution in $H^{s}$ for $s>\frac{3}{2}$ are derived. But the global existence of the small initial data solution is an open problem. If we have a deformation tensor term $F$ in deformation tensor equation $(\ref{1.1})_{2}$, the local existence of strong solution in $H^{m}$ for $m\ge3$ still holds, which is the following theorem.

$\bf{Theorem\  C}$  Assume $u_{0},F_{0}\in H^{s} (\mathbb{R}^{3})$ with $s>1+\frac{3}{2}$.
 Then, there exists a time $T=T(\|u_{0}\|_{H^{s}},\|F_{0}\|_{H^{s}})>0$ such that equations (\ref{1.1}) with $\nu>0$ have a unique strong solution $(u,F)$
 with $u,F\in C([0,T]; H^{s} (\mathbb{R}^{3}))$. Moreover, the local solution $(u,F)$ satisfies the following estimate
 \be\label{3.34}
\|u(\cdot,t)\|_{H^{s}}^{2}+\|F(\cdot,t)\|_{H^{s}}^{2}
+\int_{0}^{t}\|F(\cdot,\tau)\|_{H^{s}}^{2}+\|\nabla u(\cdot,\tau)\|_{H^{s}}^{2}\mbox{d}\tau
\leq C_{1} (\|u_{0}\|_{H^{s}}^{2}+\|F_{0}\|_{H^{s}}^{2})
\ee
for any $t\in[0,T]$.

 \begin{rem}
In Theorem C, if we only require $s>\frac {3}2$, the local existence of strong solution also holds. To have the a priori estimate (\ref{3.34}) the condition
$s>1+\frac{3}2$ is required.
 \end{rem}

To this end, we state our main results as follows:


\begin{thm}\label{them1}
Let $m \geq 3$ be an integer, assume that $(u_{0},F_{0})\in  H^{m}(\mathbb{R}^{3})$ and the initial data satisfies
\be\label{1.5}
   \|u_{0}\|_{H^{m}}+\|F_{0}\|_{H^{m}}\leq \delta_{0}\nonumber
   \ee
   for a small constant $\delta_{0}>0$.
Then, there exists a unique globally smooth solution $(u,F)$ to the Cauchy problem (\ref{1.1}) and (\ref{1.4}) satisfying
\be\label{1.6}
\|(u,F)(\cdot,t)\|_{H^{m}}^{2}+\int_{0}^{t}\|\nabla u(\cdot,\tau)\|_{H^{m}}^{2}+\|F(\cdot,\tau)\|_{H^{m}}^{2}\mbox{d}\tau\leq C_1\|u_{0},F_{0}\|_{H^{m}}^{2}\nonumber
\ee
for all $t>0$.
\end{thm}

\begin{thm}\label{2}
Under the assumption of Theorem \ref{them1}, if in addition, $(u_{0},F_{0})\in L^{1}(\mathbb{R}^{3}) \cap H^{m}(\mathbb{R}^{3})$ for $m\ge 3$,
then the smooth solution $(u,F)$ has the following optimal decay rate
    \be\label{1.7}
   \|u(t)\|_{L^{2}}^{2}+\|F(t)\|_{L^{2}}^{2}\leq C(t+1)^{-\frac{3}{2}}.\nonumber
   \ee
   \end{thm}

  The decay rate of the higher order derivative of the solution is also held.
  \begin{thm}\label{3}
 Under the assumption of Theorem \ref{2}, for any integer $j\geq0$, there exists a $T_{0}$  such that the small global-in-time
 solution satisfies
 \be\label{1.8}
   \|\nabla^{j}u(t)\|_{L^{2}}^{2}+\|\nabla^{j}F(t)\|_{L^{2}}^{2}\leq C(t+1)^{-\frac{3}{2}-j}\nonumber
   \ee
 for all $t>T_{0}$, where $C$ is a constant which depends on $j$ and the initial data.
\end{thm}

The paper unfolds as follows. In Section 2, we briefly recall some lemmas which will be used in our proof.
In Section 3, we prove the global existence of the smooth solution by the local existence result and a priori estimate.
Section 4 is devoted to the proof of Theorem \ref{2} by the classical Fourier splitting method first used by Schonbek in \cite{M}.
In Section 5, an  induction argument will be applied to get the optimal decay estimate of higher order derivative of the solution in $L^2$ norm.

Throughout this paper, $C$ denotes a generic positive constant which may be different in each occurrence. Because the specific values of the constants $\mu>0$ and $\nu>0$ are not important for our arguments, in the following parts, we take $\mu=\nu=1$.


\section{Preliminaries}
\setcounter{equation}{0}

In this preliminary section, we present some lemmas which will be used in the proof.

 In the following sections, we will apply the following commutator estimate and the product estimate of two functions, for details readers can refer to Kato-Ponce \cite{K-P} and Kenig-Ponce-Vega \cite{C-G} or Majda-Bertozzi \cite{M-B}.
  \begin{lem}Let $1<p<\infty$ and $0<s$.
Then there exists an abstract constant $C$ such that
 \begin{eqnarray}\label{b-e}
 \|[\Lambda^{s},f]g\|_{L^{p}}\leq C(\|\nabla
f\|_{L^{p_{1}}}\|\Lambda^{s-1}g\|_{L^{p_{2}}}+\|\Lambda^{s}f\|_{L^{p_{3}}}\|g\|_{L^{p_{4}}})\\
\nonumber \mx{ for } f\in \dot W^{1,p_1}\cap \dot W^{s,p_3} \mx{ and
}g\in \dot W^{s-1,p_2}\cap L^{p_4};
 \end{eqnarray}
  \begin{eqnarray}\label{fg}
 \label{26}\|\Lambda^s(f g)\|_{L^{p}}\leq C(\|f\|_{L^{p_{1}}}\|\Lambda^s g\|_{L^{p_{2}}}+\|\Lambda^s f\|_{L^{p_{3}}}\|g\|_{L^{p_{4}}})
 \end{eqnarray}
with $1<p_{2},p_{3}<\infty$ such that
\bes\frac{1}{p}=\frac{1}{p_{1}}+\frac{1}{p_{2}}=\frac{1}{p_{3}}+\frac{1}{p_{4}},\ees
where $[\Lambda^{s},f]g=\Lambda^{s}(fg)-f\Lambda^{s}g$ and
$\Lambda=(-\Delta)^{\frac{1}{2}}$.
\end{lem}

We shall use the following $L^2$ estimate of the Fourier transform of the initial datum in a ball, which can be proved by the Hausdorff-Young theorem. The readers may also refer to the Proposition 3.3 in  \cite{JQS} or \cite{ME}.
 \begin{lem}\label{lem2.5}
 Let $u_{0}\in L^{p}(\mathbb{R}^{3})$, $1\leq p<2$, then
 \be\label{S}
 \int_{S(t)}|\mathcal{F}u_{0}(\xi)|^{2}   \mbox{d}(\xi)\leq C(t+1)^{-\frac{3}{2}(\frac{2}{p}-1)},
 \ee
 where $S(t)=\{\xi\in \mathbb{R}^{3}:|\xi|\leq g(t)\}$ is ball  with $g(t)=(\frac{\gamma}{t+1})^{\frac{1}{2}}$. Here $\gamma >0$ is a constant which will be determined later, $C$ is a constant which depends on $\gamma$ and the $L^{p}$ norm of $u_{0}$.
   \end{lem}
\begin{proof}
Let $\mathcal{F}f$ denote the Fourier transform of a function $f$. For $1\leq p<2$, by the Hausdorff-Young inequality, $\mathcal{F}$ is a bounded map from $L^{p}\rightarrow L^{q}$ and
\be\label{2.8}
\|\mathcal{F} u_{0}\|_{L^{q}}\leq C\|u_{0}\|_{L^{p}},\ \frac{1}{p}+\frac{1}{q}=1.
\ee
Hence, the H\"{o}lder inequality yields
\be\label{2.9}
\int_{S(t)}|\mathcal{F}u_{0}|^{2}\mbox{d}\xi\leq
\Big(\int_{S(t)}|\mathcal{F}u_{0}|^{q}\mbox{d}\xi\Big)^{\frac{2}{q}}
\Big(\int_{S(t)}\mbox{d}\xi\Big)^{1-\frac{2}{q}}.
\ee
Combining (\ref{2.8}) and (\ref{2.9}) we have
\be
\int_{S(t)}|\mathcal{F}u_{0}|^{2}\mbox{d}\xi\leq
C\Big(\int_{S(t)}\mbox{d}\xi\Big)^{1-\frac{2}{q}},\nonumber
\ee
which implies the estimate (\ref{S}), and this completes the proof of Lemma \ref{lem2.5}.
\end{proof}


 \section{Proof of Global Existence}
\setcounter{equation}{0}

To prove the global existence of a smooth solution, we first prove the following a priori estimate.
\begin{lem}\label{lem3.1}
For an integer $m\geq3$, if there exists a small number $\delta >0$, such that
\be\label{3.31}
 \sup_{0\leq t\leq T}\|u(\cdot,t)\|_{H^{m}}+\|F(\cdot,t)\|_{H^{m}}\leq \delta,
 \ee
then, for any $t\in[0,T]$, there exists a constant $C_{1}>1$ such that
\be\label{3.25}
\|u(\cdot,t)\|_{H^{m}}^{2}+\|F(\cdot,t)\|_{H^{m}}^{2}
+\int_{0}^{t}\|F(\cdot,\tau)\|_{H^{m}}^{2}+\|\nabla u(\cdot,\tau)\|_{H^{m}}^{2}\mbox{d}\tau
\leq\|(u_{0},F_{0})\|_{H^{m}}^{2}.\nonumber
\ee
\end{lem}

\begin{proof}
We divide the a priori estimate into three steps.

\textsf{Step 1}: $L^{2}$-norms of $u$, $F$.

Taking the $L^{2}$ inner product of the equations (\ref{1.1}) with $u$ and $F$, then summing them up, one can obtain that
\be\label{1.12}
\frac{1}{2}\frac{\mathrm{d}}{{\mathrm{d}t}}(\|u\|_{L^{2}}^{2}+\|F\|_{L^{2}}^{2})
+\|F\|_{L^{2}}^{2}+\|\nabla u\|_{L^{2}}^{2}=0,
 \ee
 where we have used
$(u \cdot \nabla u,u) = (u \cdot \nabla{F_{ \cdot j}}, {F_{ \cdot j}}) = (\nabla p, u) = 0$
and  $({F_{ \cdot i}} \cdot \nabla {F_{\cdot i}}, u)+({F_{ \cdot j}}\cdot \nabla u,{F_{ \cdot j}}) = 0 $ by the divergence free conditions of $u$ and $F_{\cdot j}$.

\textsf{Step 2}: $L^{2}$-norms of $\nabla^{m}u$, $\nabla^{m}F$.

Applying the operator $\nabla^{m}$ to the both sides of (\ref{1.1}), and taking the $L^{2}$ inner product of the
resulting equations with $\nabla^{m}u$ and $\nabla^{m}F_{.j}$, respectively, adding them up and then
integrating over $\mathbb{R}^{3}$ by parts, we have
\be\label{3.3}
 &&\frac{1}{2}\frac{\mbox{d}}{\mbox{d}t}(\|\nabla^{m}u\|_{L^{2}}^{2}+\|\nabla^{m}F\|_{L^{2}}^{2})
+\|\nabla^{m}F\|_{L^{2}}^{2}+\|\nabla^{m+1}u\|_{L^{2}}^{2} \nonumber \\
&\leq&-\int_{\mathbb{R}^{3}}\nabla^{m}(u\cdot\nabla u)\cdot \nabla^{m}u\mbox{d}x
+\int_{\mathbb{R}^{3}}\nabla^{m}(F_{.i}\cdot\nabla F_{.i})\cdot \nabla^{m}u\mbox{d}x\nonumber \\
&&-\int_{\mathbb{R}^{3}}\nabla^{m}(u\cdot\nabla F_{.j})\cdot \nabla^{m}F_{.j}\mbox{d}x+
\int_{\mathbb{R}^{3}}\nabla^{m}(F_{.j}\cdot\nabla u)\cdot \nabla^{m}F_{.j}\mbox{d}x\nonumber \\
&\triangleq&\sum_{i=1}^{4}I_{i}.
\ee
In what follows, we estimate each term on the right-hand side of above equation separately.\\
For the term $I_{1}$, we obtain
\be
I_{1}&=&-\int_{\mathbb{R}^{3}}\nabla^{m}(u\cdot\nabla u)\cdot \nabla^{m}u\mbox{d}x
=-\sum_{0\leq l\leq m}C_{m}^{l}\int_{\mathbb{R}^{3}}(\nabla^{l}u\cdot\nabla^{m-l} \nabla u)\cdot \nabla^{m} u\mbox{d}x\nonumber \\
&\leq &\sum_{0\leq l\leq m}C_{m}^{l}\|\nabla^{l}u\nabla^{m-l} \nabla u\|_{L^{\frac{6}{5}}}\|\nabla^{m} u\|_{L^{6}}\mbox{d}x\nonumber.
\ee
For $0\leq l\leq [\frac{m}{2}]$, by applying the Gagliardo-Nirenberg inequality, it leads to
\be
&&\|\nabla^{l}u\nabla^{m-l}\nabla  u\|_{L^{\frac{6}{5}}}\leq C\|\nabla^{l}u\|_{L^{3}}\|\nabla^{m-l+1} u\|_{L^{2}}\nonumber \\
&\leq & C\|\Lambda^{\alpha}u\|_{L^{2}}^{1-\frac{l}{m}}\|\nabla^{m+1}u\|_{L^{2}}^{\frac{l}{m}}
\|\nabla u\|_{L^{2}}^{\frac{l}{m}}\|\nabla^{m+1} u\|_{L^{2}}^{1-\frac{l}{m}}\nonumber \\
&\leq & C\delta\|\nabla^{m+1}u\|_{L^{2}}\nonumber,
\ee
where $\alpha$ satisfies
\be
  \frac{l}{3}-\frac{1}{3}=(\frac{\alpha}{3}-\frac{1}{2})\times(1-\frac{l}{m})+(\frac{m+1}{3}-\frac{1}{2})\times\frac{l}{m}\nonumber
   \ee
with $\alpha=\frac{m-2l}{2(m-l)}\in[0,\frac{1}{2}]$.\\
However, for $[\frac{m}{2}]+1\leq l\leq m$, we have
\be
&&\|\nabla^{l}u\nabla^{m-l} \nabla u\|_{L^{\frac{6}{5}}}\leq C\|\nabla^{l}u\|_{L^{2}}\|\nabla^{m-l+1} u\|_{L^{3}}\nonumber \\
&\leq & C\|u\|_{L^{2}}^{1-\frac{l}{m+1}}\|\nabla^{m+1}u\|_{L^{2}}^{\frac{l}{m+1}}
\|\Lambda^{\alpha} u\|_{L^{2}}^{\frac{l}{m+1}}\|\nabla^{m+1} u\|_{L^{2}}^{1-\frac{l}{m+1}}\nonumber \\
&\leq & C\delta\|\nabla^{m+1}u\|_{L^{2}}\nonumber,
\ee
where $\alpha$ satisfies
\be
  \frac{m-l+1}{3}-\frac{1}{3}=(\frac{\alpha}{3}-\frac{1}{2})\times(\frac{l}{m+1})+(\frac{m+1}{3}-\frac{1}{2})\times(1-\frac{l}{m+1})\nonumber
   \ee
with $\alpha=\frac{m+1}{2l}\in(\frac{1}{2},1]$.

In both cases, we obtain
\be
 I_{1}\leq C\delta\|\nabla^{m+1}u\|_{L^{2}}^{2}.\nonumber
  \ee
For the term $I_{2}$, an application of the estimate (\ref{fg}) and integration by parts directly yields
\be
I_{2}&=&\int_{\mathbb{R}^{3}}\nabla^{m}(F_{.i}\cdot\nabla F_{.i})\cdot\nabla^{m}u\mbox{d}x\nonumber \\
&=&-\int_{\mathbb{R}^{3}}\nabla^{m-1}(F_{.i}\cdot\nabla F_{.i})\cdot\nabla^{m+1}u\mbox{d}x\nonumber \\
&\leq& \|F_{.i}\|_{L^{\infty}}\|\nabla^{m-1}\nabla F_{.i}\|_{L^{2}}\|\nabla^{m+1}u\|_{L^{2}}
+\|\nabla^{m-1}F_{.i}\|_{L^{6}}\|\nabla F_{.i}\|_{L^{3}}\|\nabla^{m+1}u\|_{L^{2}}
\nonumber \\
&\leq &C\delta(\|\nabla^{m+1}u\|_{L^{2}}^{2}+\|\nabla^{m}F_{.i}\|_{L^{2}}^{2}).\nonumber
\ee
For the term $I_{3}$, we obtain
\be
I_{3}&=&-\int_{\mathbb{R}^{3}}\nabla^{m}(u\cdot\nabla F_{.j})\cdot \nabla^{m}F_{.j}\mbox{d}x\nonumber \\
&=&-\int_{\mathbb{R}^{3}}\nabla^{m}(u\cdot\nabla F_{.j})\cdot\nabla^{m}F_{.j}\mbox{d}x
+\int_{\mathbb{R}^{3}}(u\cdot\nabla)\nabla^{m}F_{.j}\cdot\nabla^{m}F_{.j}\mbox{d}x\nonumber \\
&=&-\int_{\mathbb{R}^{3}}([\nabla^{m},u\cdot\nabla]F_{.j})\cdot\nabla^{m}F_{.j}\mbox{d}x\nonumber \\
&\leq &(\|\nabla u\|_{L^{\infty}}\|\nabla^{m}F_{.j}\|_{L^{2}}+\|\nabla^{m}u\|_{L^{6}}\|\nabla F_{.j}\|_{L^{3}})\|\nabla^{m}F_{.j}\|_{L^{2}}\nonumber \\
&\leq &C\delta(\|\nabla^{m+1}u\|_{L^{2}}^{2}+\|\nabla^{m}F_{.j}\|_{L^{2}}^{2}),\nonumber
\ee
where use has been made of the fact
\be
\int_{\mathbb{R}^{3}}(u\cdot\nabla)\nabla^{m}F_{.j}\cdot\nabla^{m}F_{.j}\mbox{d}x=0 \nonumber
\ee
and the commutator estimate (\ref{b-e}).

For the last term, by means of the estimate (\ref{fg}) it yields
\be
I_{4}&=&\int_{\mathbb{R}^{3}}\nabla^{m}(F_{.j}\cdot\nabla u)\cdot\nabla^{m}F_{.j}\mbox{d}x\nonumber \\
&\leq &C(\|F_{.j}\|_{L^{\infty}}\|\nabla^{m+1}u\|_{L^{2}}+\|\nabla^{m}F_{.j}\|_{L^{2}}\|\nabla u\|_{L^{\infty}}
)\|\nabla^{m}F_{.j}\|_{L^{2}}\nonumber \\
&\leq &C\delta(\|\nabla^{m+1}u\|_{L^{2}}^{2}+\|\nabla^{m} F_{.j}\|_{L^{2}}^{2}).\nonumber
\ee
Substituting the estimates $I_{1}-I_{4}$ into (\ref{3.3}), one has the key estimate by choosing $\delta$ small enough.
\be\label{1.9}
\frac{\mbox{d}}{\mbox{d}t}(\|\nabla^{m}u\|_{L^{2}}^{2}+\|\nabla^{m}F\|_{L^{2}}^{2})
+\|\nabla^{m}F\|_{L^{2}}^{2}+\|\nabla^{m+1}u\|_{L^{2}}^{2}\leq 0.
\ee


\textsf{Step 3}: Conclusion

Summing up (\ref{1.12}) and (\ref{1.9}), we thereby obtain
\be\label{3.33}
\frac{\mbox{d}}{\mbox{d}t}(\|u\|_{H^{m}}^{2}+\|F\|_{H^{m}}^{2})
+\|F\|_{H^{m}}^{2}+\|\nabla u\|_{H^{m}}^{2}\leq 0.\nonumber
\ee
Integrating the above inequality directly in time leads to
\be
\|u(\cdot,t)\|_{H^{m}}^{2}+\|F(\cdot,t)\|_{H^{m}}^{2}
+\int_{0}^{t}\|F(\cdot,\tau)\|_{H^{m}}^{2}+\|\nabla u(\cdot,\tau)\|_{H^{m}}^{2}\mbox{d}\tau
\leq \|(u_{0},F_{0})\|_{H^{m}}^{2},\nonumber
\ee
 we thus finish the proof of Lemma \ref{lem3.1}.
\end{proof}


Combining the local existence Theorem \ref{them1} and the a priori estimate Lemma \ref{lem3.1}, we will complete the proof of the global existence of the smooth solution by a continuous extending argument.

\vspace{0.2in}
\textsf{Proof of Theorem \ref{them1}}
\begin{proof}
Assume
\be\label{3.35}
E_{0}:=\|u_{0}\|_{H^{m}}+\|F_{0}\|_{H^{m}}<\delta/\sqrt{C_{1}},
\ee
where $\delta$ is defined in Lemma \ref{lem3.1}. By choosing $\delta_0=\delta/\sqrt{C_{1}}$, we can prove there exists a global-in-time solution to the system (\ref{1.1}).
As the initial data satisfies $E_{0}<\delta/\sqrt{C_{1}}$, then according to Theorem C there exists a positive
constant $T_{1}> 0$, such that the smooth solution of (\ref{1.1}) and (\ref{1.4}) exists on $[0,T_{1}]$ and the following estimate holds.
\be
\|u(\cdot,t)\|_{H^{m}}^{2}+\|F(\cdot,t)\|_{H^{m}}^{2}
+\int_{0}^{t}\|F(\cdot,\tau)\|_{H^{m}}^{2}+\|\nabla u(\cdot,\tau)\|_{H^{m}}^{2}\mbox{d}\tau
\leq C_{1}E_{0}^{2}
\ee
for $t\in[0,T_{1}]$.
It implies
\be\label{3.39}
E_{1}:=\sup_{0\leq t\leq T_{1}}\|(u,F)(\cdot,t)\|_{H^{m}}\leq \sqrt{C_{1}}E_{0}<\delta.
\ee
Thus the solution satisfies the a priori estimate (\ref{3.31}), by Lemma \ref{lem3.1} and (\ref{3.35}) we get
\be\label{3.36}
E_{1}\leq E_{0}<\delta \sqrt{C_{1}}.
\ee
Therefore by Theorem C the initial problem (\ref{1.1}) for $t\geq T_{1}$, with the initial data $(u,F)(x,T_{1})$, has again a unique local solution $(u,F)$ satisfying
\be
\|u(\cdot,t)\|_{H^{m}}^{2}+\|F(\cdot,t)\|_{H^{m}}^{2}
+\int_{T_1}^{t}\|F(\cdot,\tau)\|_{H^{m}}^{2}+\|u(\cdot,\tau)\|_{H^{m+1}}^{2}\mbox{d}\tau
\leq C_{1}(\|u(\cdot,T_{1})\|_{H^{m}}^{2}+\|F(\cdot,T_{1})\|_{H^{m}}^{2}),\nonumber
\ee
for $t\in[T_{1},2T_{1}]$.
Combining  this with (\ref{3.36}), it yields
\be\label{3.38}
\sup_{T_{1}\leq t\leq 2T_{1}}\|(u,F)(\cdot,t)\|_{H^{m}}\leq \sqrt{C_{1}}E_{1}<\delta.
\ee
Then by (\ref{3.39}), (\ref{3.38}) and Lemma \ref{lem3.1}, it gives rise to
\be
E_{2}:=\sup_{0\leq t\leq 2T_{1}}\|(u,F)(\cdot,t)\|_{H^{m}}\leq E_{0}<\delta/\sqrt{C_{1}}.\nonumber
\ee
Therefore, we can repeat the same argument as above for $0\leq t\leq nT_{1}, \ n=3,4,\cdots $ and finally obtain
the global existence of the smooth solution for the system (\ref{1.1}).\par

\end{proof}

 \section{Proof of Theorem \ref{2}}
\setcounter{equation}{0}

In this section we prove the decay rate of the smooth solution to the equations (\ref{1.1}) in $L^{2}$ space. For the convenience of presentation, we denote the Fourier transform of $f$ by $\mathcal{F}f$ or $\hat{f}$ in the subsequences.

In Section 3, we have already obtained
\be\label{1}
\frac{\mathrm{d}}{{\mathrm{d}t}}(\|u\|_{L^{2}}^{2}+\|F\|_{L^{2}}^{2})+
\|\nabla u\|_{L^{2}}^{2}+\|F\|_{L^{2}}^{2}=0.
 \ee
Applying Plancherel's theorem to (\ref{1}) it yields
\be\label{4.19}
\frac{\mbox{d}}{\mbox{d}t}\int_{\mathbb{R}^{3}}(|\hat{u}(\xi)|^{2}+|\hat{F}(\xi)|^{2})\mbox{d}\xi=
-\int_{\mathbb{R}^{3}}(|\xi|^{2}|\widehat{u}(\xi)|^{2}+|\widehat{F}(\xi)|^{2})\mbox{d}\xi.\nonumber
\ee
By decomposing the frequency domain into two time-dependent subsets, it yields
\bes
&&\frac{\mbox{d}}{\mbox{d}t}\int_{\mathbb{R}^{3}}(|\hat{u}(\xi)|^{2}+|\hat{F}(\xi)|^{2})\mbox{d}\xi\nonumber\\
&\leq& -\int_{|\xi|\geq g(t) }g(t)^{2}|\hat{u}(\xi)|^{2}\mbox{d}\xi
-\int_{|\xi|\leq g(t) }|\xi|^{2}|\hat{u}(\xi)|^{2}\mbox{d}\xi-\int_{\mathbb{R}^{3}}|\hat{F}(\xi)|^{2}\mbox{d}\xi\nonumber\\
&=& -\int_{\mathbb{R}^{3} }g(t)^{2}|\hat{u}(\xi)|^{2}\mbox{d}\xi
+\int_{|\xi|\leq g(t)}g(t)^{2}|\hat{u}(\xi)|^{2}\mbox{d}\xi-\int_{\mathbb{R}^3}|\hat{F}(\xi)|^{2}\mbox{d}\xi,
\ees
where $g(t)$ is defined in Lemma \ref{lem2.5} and $\gamma$ is a constant to be determined later. There exists a time $T_0>0$ such that, when $t>T_0$, one has
\be\label{4.20}
\frac{\mbox{d}}{\mbox{d}t}\int_{\mathbb{R}^{3}}|\hat{u}(\xi)|^{2}+|\hat{F}(\xi)|^{2}\mbox{d}\xi+
\frac{\gamma}{1+t}\int_{\mathbb{R}^{3} }|\hat{u}(\xi)|^{2}+|\hat{F}(\xi)|^2\mbox{d}\xi\le
\frac{\gamma}{1+t}\int_{|\xi|\leq g(t)}|\hat{u}(\xi)|^{2}+|\hat{F}(\xi)|^2\mbox{d}\xi.
\ee
Multiplying  (\ref{4.20}) by the integrating factor $(t+1)^{\gamma}$, it follows that
\be\label{4.24}
\frac{\mbox{d}}{\mbox{d}t}\Big((t+1)^{\gamma}(\|u(t)\|_{L^{2}}^{2}+\|F(t)\|_{L^{2}}^{2})\Big)\leq \gamma(t+1)^{\gamma-1}\int_{|\xi|\leq g(t)}(|\hat{u}(\xi)|^{2}+|\hat{F}(\xi)|^{2})\mbox{d}\xi.
\ee

To finish the proof, we prove the estimates of $|\hat{u}(\xi)|$ and $|\hat{F}(\xi)|$ as follows.
\begin{lem}\label{lem3.3}
Let $(u,F)$ be a smooth solution to the Cauchy problem (\ref{1.1}) with the small initial data $(u_{0},F_{0})\in L^{1} \cap H^{m}$, $m\ge 3$. Then there exist
\be\label{4.22}
|\widehat{u}(\xi,t)|\leq C(|\widehat{u_{0}}(\xi)|+\frac{1}{|\xi|})
\ee
and
\be\label{4.23}
|\widehat{F}(\xi,t)|\leq C(|\widehat{F_{0}}(\xi)|+|\xi|).
\ee
\end{lem}

\begin{proof}
Taking the Fourier transform of the equations (\ref{1.1}) we have
\be\label{4.1}
\widehat{u_{t}}(\xi,t)+|\xi|^{2}\widehat{u}(\xi,t)=H(\xi,t),
\ee
where $H(\xi,t)=-\widehat{u\cdot \nabla u}(\xi,t)-\widehat{\nabla p}(\xi,t)+\widehat{F\cdot\nabla F_{.i}}(\xi,t) $
and
\be\label{4.2}
\widehat{F_{t}}(\xi,t)+\widehat{F}(\xi,t)=G(\xi,t),
\ee
where $G(\xi,t)=-\widehat{u\cdot \nabla F}(\xi,t)+\widehat{ F\cdot\nabla u}(\xi,t)$.\\
Multiplying (\ref{4.1}) and (\ref{4.2}) by the integrating factor $\mbox{e}^{|\xi|^{2}t}$ and $\mbox{e}^{t}$ respectively, we have
\be\label{4.3}
\frac{\mbox{d}}{\mbox{d}t}(\mbox{e}^{|\xi|^{2}t}\widehat{u}(\xi,t))\leq \mbox{e}^{|\xi|^{2}t}H(\xi,t)
\ee
and
\be\label{4.15}
\frac{\mbox{d}}{\mbox{d}t}(\mbox{e}^{t}\widehat{F}(\xi,t))\leq
\mbox{e}^{t}G(\xi,t).
\ee
Integrating (\ref{4.3}) and (\ref{4.15}) in time from $0$ to $t$, it arrives at
\be\label{4.4}
\widehat{u}(\xi,t)\leq \mbox{e}^{-|\xi|^{2}t}\widehat{u_{0}}(\xi,t)+\int_{0}^{t}\mbox{e}^{-|\xi|^{2}(t-\tau)}H(\xi,\tau)\mbox{d}\tau
\ee
and
\be\label{4.16}
\widehat{F}(\xi,t)\leq \mbox{e}^{-t}\widehat{F_{0}}(\xi,t)+\int_{0}^{t}\mbox{e}^{-(t-\tau)}G(\xi,\tau)\mbox{d}\tau.\nonumber
\ee
Now, we derive the estimates for $H(\xi,t)$ and $G(\xi,t)$.
Taking the divergence operator on the first equation of (\ref{1.1}) and by using the divergence free condition of $u$ and $F$ one has
\be\label{4.5}
\triangle p=-\nabla\cdot\mbox{div}(u\otimes u)+\nabla\cdot\mbox{div}(F\otimes F).\nonumber
\ee
Since the Fourier transform is bounded map from $L^{1}$ to $L^{\infty}$, it leads to
\be\label{4.6}
|\widehat{\nabla p}(\xi,t)|&\leq &|\xi||\widehat{p}(\xi,t)|\leq |\xi|(\|u(t)u(t)\|_{L^{1}}+\|F(t)F(t)\|_{L^{1}})\nonumber \\
&\leq& C|\xi|(\|u(t)\|_{L^{2}}^{2}+\|F(t)\|_{L^{2}}^{2}).
\ee
Similarly, for the convective terms, we also have
\be\label{4.7}
|\widehat{u\cdot \nabla u}(\xi,t)|\leq C |\xi|\|u(t)\|_{L^{2}}^{2}
\ee
and
\be\label{4.8}
|\widehat{F\cdot \nabla F}(\xi,t)|\leq C |\xi|\|F(t)\|_{L^{2}}^{2},
\ee
and the following estimates
\be\label{4.9}
|\widehat{u\cdot \nabla F}(\xi,t)|\leq C |\xi|(\|u(t)\|_{L^{2}}^{2}+\|F(t)\|_{L^{2}}^{2})
\ee
and
\be\label{4.10}
|\widehat{F\cdot \nabla u}(\xi,t)|\leq C |\xi|(\|u(t)\|_{L^{2}}^{2}+\|F(t)\|_{L^{2}}^{2}).
\ee
Combining the estimates (\ref{4.6})-(\ref{4.8}) together, we get
\be\label{4.12}
|H(\xi,t)|\leq C|\xi|(\|u(t)\|_{L^{2}}^{2}+\|F(t)\|_{L^{2}}^{2}).
\ee
Combining the estimates (\ref{4.9})-(\ref{4.10}), we obtain
\be\label{4.13}
|G(\xi,t)|\leq C|\xi|(\|u(t)\|_{L^{2}}^{2}+\|F(t)\|_{L^{2}}^{2}).\nonumber
\ee
Inserting $|H(\xi,t)|$ into (\ref{4.4}) and using the boundedness of $L^{2}$ norms of the solution, we deduce
\be\label{4.14}
|\widehat{u}(\xi,t)|&\leq&|\widehat{u_{0}}(\xi)|+\frac{C}{|\xi|}(\|u_{0}\|_{L^{2}}^{2}+\|F_{0}\|_{L^{2}}^{2})(1-\mbox{e}^{-|\xi|^{2}t})\nonumber \\
&\leq & C(|\widehat{u_{0}}(\xi)|+\frac{1}{|\xi|}).\nonumber
\ee
Using a similar argument, we have
\be\label{4.17}
|\widehat{F}(\xi,t)|&\leq&|\widehat{F_{0}}(\xi)|+
C|\xi|(\|u_{0}\|_{L^{2}}^{2}+\|F_{0}\|_{L^{2}}^{2})(1-\mbox{e}^{-t})\nonumber \\
&\leq & C(|\widehat{F_{0}}(\xi)|+|\xi|).\nonumber
\ee
We thus derive the estimates of $|\hat{u}(\xi)|$ and $|\hat{F}(\xi)|$.
\end{proof}

Putting (\ref{4.22}) and (\ref{4.23}) into the right-hand side of (\ref{4.24}) and applying Lemma \ref{lem2.5}, it follows
\be\label{4.25}
\frac{\mbox{d}}{\mbox{d}t}\Big((t+1)^{\gamma}(\|u(t)\|_{L^{2}}^{2}+\|F(t)\|_{L^{2}}^{2})\Big)
&\leq & C(t+1)^{\gamma-1}\int_{|\xi|\leq g(t)}(|\hat{u_{0}}(\xi)|^{2}+|\hat{F_{0}}(\xi)|^{2})\mbox{d}\xi\nonumber\\
&+& C(t+1)^{\gamma-1}\int_{|\xi|\leq g(t)}\frac{1}{|\xi|^{2}}\mbox{d}\xi+ C(t+1)^{\gamma-1}\int_{|\xi|\leq g(t)}|\xi|^{2}\mbox{d}\xi\nonumber\\
&\leq & C(t+1)^{\gamma-1-\frac{3}{2}}+C(t+1)^{\gamma-1-\frac{1}{2}}+C(t+1)^{\gamma-1-\frac{5}{2}}.\nonumber
\ee
Integrating the above inequality in time from $0$ to $t$ leads to
\be
\|u(t)\|_{L^{2}}^{2}+\|F(t)\|_{L^{2}}^{2}\leq C\Big((t+1)^{-\gamma}+(t+1)^{-\frac{3}{2}}+C(t+1)^{-\frac{1}{2}}+C(t+1)^{-\frac{5}{2}}\Big).\nonumber
\ee
By choosing $\gamma>\frac12$, we obtain
\be\label{4.26}
\|u(t)\|_{L^2}^{2}+\|F(t)\|_{L^2}^{2}\leq C(t+1)^{-\frac{1}{2}}.
\ee
Again inserting the above estimate (\ref{4.26}) of $\|u(t)\|_{L^2}^{2}+\|F(t)\|_{L^2}^{2}$ into the estimate (\ref{4.12}), it follows
\be\label{1.3}
\int_{0}^{t}\mbox{e}^{-|\xi|^{2}(t-\tau)}|H(\xi,\tau)|\mbox{d}\tau&\leq& C|\xi|\int_{0}^{t}(\tau+1)^{-\frac{1}{2}}\mbox{d}\tau\nonumber\\
&\leq& C|\xi|\Big((t+1)^{\frac{1}{2}}-1)\Big)\nonumber\\
&\leq& C (t+1)^{-\frac{1}{2}}  \Big((t+1)^{\frac{1}{2}}\Big)\le C,
\ee
if $|\xi|$ is in the ball $S(t)$ defined in Lemma \ref{lem2.5}.
Putting (\ref{1.3}) into (\ref{4.4}), we get $|\widehat{u}(\xi,t)|\leq C(|\widehat{u_{0}}(\xi)|+1)$.
Arguing similarly, we obtain $|\widehat{F}(\xi,t)|\leq C(|\widehat{F_{0}}(\xi)|+1)$.
Inserting these estimates of $\widehat{u}(\xi,t)$ and $\widehat{F}(\xi,t)$ into (\ref{4.24}) and by Lemma \ref{lem2.5} we have
\be
\frac{\mbox{d}}{\mbox{d}t}\Big((t+1)^{\gamma}(\|u(t)\|_{L^2}^{2}+\|F(t)\|_{L^2}^{2})\Big)
\leq \gamma(t+1)^{\gamma-1}(t+1)^{-\frac{3}{2}}.\nonumber
\ee
Integrating the above estimate in time and choosing $\gamma>\frac32$, it leads to
\be
\|u(t)\|_{L^2}^{2}+\|F(t)\|_{L^2}^{2}\leq C(t+1)^{-\frac{3}{2}},\nonumber
\ee
which completes the proof of Theorem \ref{2}.


 \section{Proof of Theorem \ref{3}}
\setcounter{equation}{0}

This section is devoted to showing the higher order derivative's optimal decay estimate of a smooth solution to the equations (\ref{1.1}) in $L^2$ norm.

\begin{proof}
As usual, we denote $S(t)=\{\xi\in \mathbb{R}^{3}:|\xi|\leq f(t)\}$, with $f(t)=(\frac{\gamma}{t+1})^{\frac{1}{2}}$,
where $\gamma$ is a constant to be determined later. For the order $m+1$ derivative term, by the Fourier-splitting method again, it is deduced as follows
\be\label{4.29}
\|\Lambda^{m+1}u\|_{L^{2}}^{2}&=&\int_{\mathbb{R}^{3}}|\xi|^{2}|\mathcal{F}\Lambda^{m}u(\xi,t)|^{2}\mbox{d}\xi\nonumber \\
&\geq &\int_{|\xi|\geq  f(t)}|\xi|^{2}|\mathcal{F}\Lambda^{m}u(\xi,t)|^{2}\mbox{d}\xi\nonumber \\
&\geq &f^{2}(t)\|\Lambda^{m}u\|_{L^{2}}^{2}-f^{2}(t)\int_{S(t)}|\mathcal{F}\Lambda^{m}u(\xi,t)|^{2}\mbox{d}\xi\nonumber \\
&\geq &f^{2}(t)\|\Lambda^{m}u\|_{L^{2}}^{2}-f^{4}(t)\int_{\mathbb{R}^{3}}|\mathcal{F}\Lambda^{m-1}u(\xi,t)|^{2}\mbox{d}\xi,
\ee
where $m\ge 1$ is an integer.

Inserting the estimate (\ref{4.29}) into (\ref{1.9}), it follows that for $t>T_0$ with some a $T_0>0$
\be\label{4.31}
 &&\frac{\mbox{d}}{\mbox{d}t}(\|\Lambda^{m}u\|_{L^{2}}^{2}+\|\Lambda^{m}F\|_{L^{2}}^{2})
+\frac{\gamma}{t+1}\|\Lambda^{m}F\|_{L^{2}}^{2}+\frac{\gamma}{t+1}(t)\|\Lambda^{m}u\|_{L^{2}}^{2} \nonumber \\
&&\leq \Big(\frac{\gamma}{t+1}\Big)^{2}(\|\Lambda^{m-1}u\|_{L^{2}}^{2} + \|\Lambda^{m-1}F\|_{L^{2}}^{2}).
\ee
If $m=1$, multiplying both sides of the inequality (\ref{4.29}) by $(t+1)^\gamma$ one has
\be\label{a1}
\frac{\mbox{d}}{\mbox{d}t}\Big((t+1)^{\gamma}(\|\Lambda u\|_{L^{2}}^{2}+\|\Lambda F\|_{L^{2}}^{2})\Big)&\leq& C{t+1}^{\gamma-2}(\|u\|_{L^{2}}^{2} + \|F\|_{L^{2}}^{2})\\ \nonumber
&\le& C(t+1)^{\gamma-2-\frac32}.
\ee
Integrating the inequality (\ref{a1}) from $T_0$ to $t$ one has
\be
&&(t+1)^{\gamma}(\|\Lambda u\|_{L^{2}}^{2}+\|\Lambda F\|_{L^{2}}^{2})\nonumber\\
&\leq &(T_{0}+1)^{\gamma}(\|\Lambda u(T_{0})\|_{L^{2}}^{2}+\|\Lambda F(T_{0})\|_{L^{2}}^{2})+C(t+1)^{\gamma-1-\frac32}.\nonumber
\ee
Therefore we can obtain by choosing $\gamma>\frac52$
\be
\|\Lambda u\|_{L^{2}}^{2}+\|\Lambda F\|_{L^{2}}^{2}&\leq & C(t+1)^{-\frac32-1}.
\ee

To finish the proof of Theorem \ref{3}, we use the argument of induction by $m$.\\
Assume
\be\label{4}
\|\Lambda^{m-1}u\|_{L^{2}}^{2}+\|\Lambda^{m-1}F\|_{L^{2}}^{2}\leq C_{m-1}(t+1)^{-\frac32-(m-1)}.
\ee
 After inserting (\ref{4}) into (\ref{4.31}), and multiplying $(t+1)^{\gamma}$ on the both sides of the resulting inequality, we derive
 \be\label{4.32}
\frac{\mbox{d}}{\mbox{d}t}\Big((t+1)^{\gamma}(\|\Lambda^{m}u\|_{L^{2}}^{2}+\|\Lambda^{m}F\|_{L^{2}}^{2})\Big)\leq \gamma^2C_{m-1}(t+1)^{\gamma-\frac32-(m-1)-2}.\nonumber
\ee
Integrating the above inequality in time from $T_{0}$ to $t$ we get
\be
&&(t+1)^{\gamma}(\|\Lambda^{m}u\|_{L^{2}}^{2}+\|\Lambda^{m}F\|_{L^{2}}^{2})\nonumber\\
&\leq &(T_{0}+1)^{\gamma}(\|\Lambda^{m}u(T_{0})\|_{L^{2}}^{2}+\|\Lambda^{m}F(T_{0})\|_{L^{2}}^{2})+\gamma^2C_{m-1}(t+1)^{\gamma-\frac32-(m-1)-1}.\nonumber
\ee
Similarly, by choosing $\gamma>\frac32+m$ we obtain
\bes
\|\Lambda^{m}u\|_{L^{2}}^{2}+\|\Lambda^{m}F\|_{L^{2}}^{2}\leq C_{m}(t+1)^{-\frac{3}{2}-m}.
\ees
We thus complete the proof of Theorem \ref{3}.
\end{proof}


\vspace{0.4cm}

%
%
%

 \textbf{Acknowledgements} The research of B Yuan
was partially supported by the National Natural Science Foundation
of China (No. 11471103).


\end{document}